\documentclass[11pt,a4paper]{article}
\usepackage[top=1in, bottom=1.25in, left=1in, right=1in]{geometry}
\usepackage[latin1]{inputenc}
\usepackage[english]{babel}
\usepackage{graphicx,epstopdf}
\usepackage{caption,subcaption,float}
\usepackage{amsthm}
\usepackage{color}
\usepackage{tikz}
\usetikzlibrary{positioning,shapes}
\usepackage{relsize}
\usepackage[ruled, linesnumberedhidden,vlined]{algorithm2e} 
\usepackage[round,sort,comma]{natbib}
\usepackage{aliascnt}
\usepackage{hyperref}
 \usepackage{multirow}
\hypersetup{
    colorlinks=true,       
    linkcolor=blue,          
    citecolor=olive,       
    filecolor=magenta,      
    urlcolor=blue          
}

\usetikzlibrary{matrix}
\usepackage{amsmath,amssymb,dsfont}
\title{Conjugacy stability of parabolic subgroups of Artin-Tits groups of spherical type}
\author{Matthieu Calvez, Bruno A. Cisneros de la Cruz, Mar\'ia Cumplido}
\date{\today }

\begin{document}

\maketitle

% ------- Theorem styles -------
\theoremstyle{plain}
\newtheorem{theorem}{Theorem}

\newaliascnt{lemma}{theorem}
\newtheorem{lemma}[lemma]{Lemma}
\aliascntresetthe{lemma}
\providecommand*{\lemmaautorefname}{Lemma}

\newaliascnt{proposition}{theorem}
\newtheorem{proposition}[proposition]{Proposition}
\aliascntresetthe{proposition}
\providecommand*{\propositionautorefname}{Proposition}

\newaliascnt{corollary}{theorem}
\newtheorem{corollary}[corollary]{Corollary}
\aliascntresetthe{corollary}
\providecommand*{\corollaryautorefname}{Corollary}

\newaliascnt{conjecture}{theorem}
\newtheorem{conjecture}[conjecture]{Conjecture}
\aliascntresetthe{conjecture}
\providecommand*{\conjectureautorefname}{Conjecture}

\theoremstyle{remark}

\newaliascnt{claim}{theorem}
\newaliascnt{remark}{theorem}
\newtheorem{claim}[claim]{Claim}
\newtheorem{remark}[remark]{Remark}
\newaliascnt{notation}{theorem}
\newtheorem{notation}[notation]{Notation}
\aliascntresetthe{notation}
\providecommand*{\notationautorefname}{Notation}

\aliascntresetthe{claim}
\providecommand*{\claimautorefname}{Claim}

\aliascntresetthe{remark}
\providecommand*{\remarkautorefname}{Remark}

\newtheorem*{claim*}{Claim}
\theoremstyle{definition}

\newaliascnt{definition}{theorem}
\newtheorem{definition}[definition]{Definition}
\aliascntresetthe{definition}
\providecommand*{\definitionautorefname}{Definition}

\newaliascnt{example}{theorem}
\newtheorem{example}[example]{Example}
\aliascntresetthe{example}
\providecommand*{\exampleautorefname}{Example}

\newcommand{\myref}[2]{\hyperref[#1]{#2~\ref*{#1}}}

%|<------------------------------------------------------------------------>|

\begin{abstract}
We give a complete classification of conjugacy stable parabolic subgroups of Artin-Tits groups of spherical type. This answers a question posed by Ivan Marin and generalizes a theorem obtained by Juan Gonz\'alez-Meneses in the specific case of Artin braid groups.
\end{abstract}

\section{Introduction}

\noindent
Let $S$ be a finite set. A {\it Coxeter matrix} over $S$ is a symmetric square matrix $M = (m_{s,t})_{s,t\in S}$ indexed by the elements of $S$, such that $m_{s,s}=1$, and $m_{s,t} \in \{ 2,3,4, \ldots, \infty \}$ for all $s,t\in S$, $s\neq t$.  Such a Coxeter matrix is usually represented by its {\it Coxeter graph}, denoted by $\Gamma = \Gamma_S =\Gamma(M)$. This  is a labelled graph whose set of vertices is $S$, in which two distinct vertices $s$ and $t$ are connected by an edge if $m_{s,t}\geqslant 3$; if in addition $m_{s,t}\geqslant 4$, the corresponding edge wears the label  $m_{s,t}$.  

\medskip
\noindent
The {\it  Artin-Tits system} of $\Gamma$ is the pair $(A,S)$, where $A=A_{\Gamma}$ is the group
\begin{eqnarray*}
%\label{Presentation} 
A_{\Gamma} = \left\langle S \ \left|  	\begin{array}{ll} 		
							\Pi(s,t;m_{s,t}) = \Pi(t,s;m{s,t}) 	& \text{for all } s,t\in S, \; s\neq t, \; m_{s,t}\neq \infty 
						\end{array}\right.\right\rangle , 
\end{eqnarray*}
where, for $m\geqslant 2$, $$\Pi(a,b;m)=\begin{cases}  (ab)^k & {\text{if $m=2k$}},\\  (ab)^ka &  {\text{if $m=2k+1$.}}\\ \end{cases}$$
The group $A_{\Gamma}$ is called the {\it Artin-Tits group} of $\Gamma$; sometimes we shall also use the notation $A_{S}$ to refer to this group. If we add to the presentation of $A_\Gamma$ the relations $s^2=1$, for every $s\in S$, we obtain the \emph{Coxeter group} associated to $A_\Gamma$. When this group is finite we say that $A_\Gamma$ has \emph{spherical type}. By extension, we say that~$\Gamma$ is of spherical type if~$A_{\Gamma}$ has spherical type. 
$A_{\Gamma}$ is called {\it  irreducible} if the graph $\Gamma$ is connected and {\it reducible} otherwise. Notice that if $\Gamma_1,\cdots,\Gamma_r$ are the connected components of~$\Gamma$, then $A_\Gamma=A_{\Gamma_1}\times \cdots \times A_{\Gamma_r}$.  We recall Coxeter's classification \citep{Coxeter} of connected Coxeter graphs of spherical type (hence of irreducible Artin-Tits groups of spherical type) in \autoref{F:Coxeter}. The name of the graph will be used to refer to the corresponding Artin-Tits group; for instance the Artin-Tits group of type $E_7$ is the Artin-Tits group of the graph~$E_7$.  

\begin{figure}[hbt]
\centering
\includegraphics[scale=1]{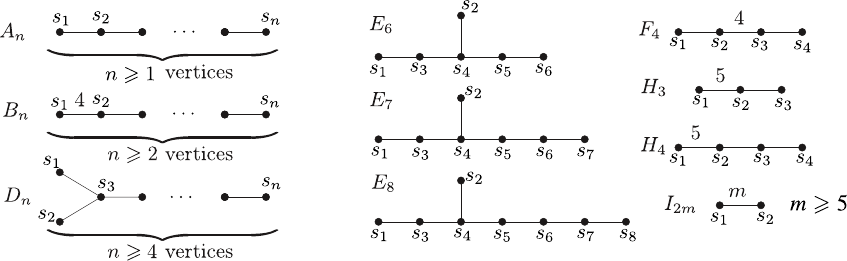}
\caption{Connected Coxeter graphs of spherical type with a specific enumeration of the vertices.}\label{F:Coxeter}
\end{figure}

\medskip\noindent
Let $X$ be a subset of $S$. The {\it  standard parabolic subgroup} associated to~$X$ is the subgroup of~$A_{\Gamma}$ generated by $X$ and denoted  by~$A_X$.  Consider the subgraph $\Gamma_X$ of $\Gamma=\Gamma_S$ generated by~$X$ (the set of vertices is $X$ and the edges are exactly the edges of~$\Gamma_S$ which connect two vertices in~$X$). 
It is known \citep{VanderLek} that $(A_{X}, X)$ is the Artin-Tits system of~$\Gamma_X$. A {\it parabolic subgroup} is a subgroup~$P$ conjugate to some standard parabolic subgroup~$A_X$. Note that~$P$ and~$A_X$ are isomorphic; if~$A_X$ is irreducible of spherical type, the type of~$P$ is the name of the graph~$\Gamma_X$ in \autoref{F:Coxeter}. 

\medskip\noindent
The flagship example of an Artin-Tits group of spherical type is the braid group on $n$ strands~$\mathcal B_n$ ($n\geqslant 2$) \citep{Artin}. It is associated to the Coxeter graph $A_{n-1}$ depicted in \autoref{F:Coxeter}; the corresponding Coxeter group is the symmetric group $\mathfrak S_n$. We recall that each generator $s_i$ is the crossing of the strands in the positions $i$ and $i+1$. Let~$m$ and~$n$ be two positive integers such that $2\leqslant m\leqslant n$. Considering only the $m-1$ first vertices of the graph $A_{n-1}$ furnishes a fundamental example of a standard (irreducible) parabolic subgroup: the braid group~$\mathcal B_m$ embedded in~$\mathcal B_n$ by adding $n-m$ straight strands to any $m$-strand braid.

\bigskip\noindent
It was shown in \citep{GMParabolic} that the above embedding 
$\mathcal B_m\hookrightarrow \mathcal B_n$ (for $2\leqslant m<n$) does not merge conjugacy classes, i.e. if two $m$-strand braids are conjugate in the $n$-strand braid group,  they must already be conjugate as $m$-strand braids.  

\medskip\noindent
Motivated by the latter result, Ivan Marin asked some years ago whether \emph{standard} parabolic subgroups of irreducible Artin-Tits groups of spherical type are conjugacy stable. A (non-trivial) proper subgroup~$H$ of a group~$G$ is said to be \emph{conjugacy stable} if any two elements of~$H$ which are conjugated in~$G$ must be conjugated through an element of~$H$; this is equivalent to saying that the conjugacy classes of~$H$ do not merge in~$G$. It is an easy exercise to check that conjugacy stability is preserved under subgroup conjugation; therefore Marin's question actually covers all parabolic subgroups of irreducible Artin-Tits groups of spherical type. 

\bigskip\noindent
Suppose now that $A_S$ is a reducible Artin-Tits group of spherical type, expressed as the direct product $A_S=A_{S_1}\times \ldots \times A_{S_r}$, where $r>1$ and each~$A_{S_i}$ is non-trivial and irreducible.  For a subset $X\subsetneq S$, we can consider $X_i=X\cap S_i$ ($i=1,\ldots,r$) and decompose $A_X$ as a direct product of parabolic subgroups $A_X=A_{X_1}\times \ldots \times A_{X_r}$ --notice that~$A_{X_i}$ might be trivial (when~$X_i$ is empty) or reducible. Since elements in distinct components of~$A_S$ commute pairwise, $A_X$ is conjugacy stable in~$A_S$ if and only if $A_{X_i}$ is conjugacy stable in~$A_{S_i}$ for all~$i$. 

\medskip\noindent
In view of the above remarks, the following, which is our main result, allows to decide the conjugacy stability of any given parabolic subgroup of any Artin-Tits group of spherical type:

\begin{theorem}\label{mainT}
Let $A_{\Gamma}=A_S$ be an irreducible Artin-Tits group of spherical type and let ${\emptyset\neq X\subsetneq S}$. 
%and $\alpha \in A_S$, 
%and consider the proper parabolic subgroup $P:= \alpha^{-1}A_X\alpha$ of $A_S$. 
\begin{itemize}
\item[(1)] If $A_X$ is irreducible, $A_X$ is conjugacy stable in $A_S$ except in the following cases: 
\begin{itemize}
%\item $A_X$ is of type $D_4$ and $A_S$ is of type $E_6,E_7,E_8$ or $D_n$ ($n$ odd, $n \geqslant 5$),
\item[(a)] $A_X$ is of type $D_5$ and $A_S$ is of type $E_6,E_7$ or $E_8$,
%\item $A_X$ is of type $D_6$ and $A_S$ is of type $E_7,E_8$ or $D_n$ ($n$ odd, $n\geqslant 7$),
\item[(b)] $A_X$ is of type $D_7$ and $A_S$ is of type $E_8$,
\item[(c)] $A_X$ is of type $E_7$ and $A_S$ is of type $E_8$,
\item[(d)] $A_X$ is of type $D_{2k}$,
\item[(e)] $A_X$ is of type $H_3$ and $A_S$ is of type $H_4$. 
\end{itemize}
\item[(2)] If $A_X$ is reducible, $A_X$ is \emph{not} conjugacy stable in $A_S$ except in the following cases:
\begin{itemize}
\item[(a)] $A_S$ is of type $B_n$ ($n\geqslant 3$) and $A_X=A_{\{s_1\}}\times A_{Z}$, with $Z\subset \{s_3,\ldots,s_n\}$ and $A_{Z}$ irreducible. 
\item[(b)] $A_S$ is of type~$F_4$. 
%{\color{green}and $A_X=A_{X_1}\times A_{X_2}$ has two irreducible components, satisfying $X_1\subset \{s_1,s_2\}$ and $X_2\subset \{s_3,s_4\}$ Matt: I think this is automatic for irreducible $X$. Could remove.} 
\end{itemize}
\end{itemize}
\end{theorem}

\noindent
Gonz\'alez-Meneses' proof in the specific case of braids relies heavily on the identification between braids and mapping classes of punctured disks: Birman-Lubotzky-McCarthy's {\it Canonical Reduction Systems} of mapping classes play a fundamental role.  Although more combinatorial in spirit, our approach was inspired by Gonz\'alez-Meneses': instead of the Canonical Reduction System, we use the {\it parabolic closure} of elements of Artin-Tits groups of spherical type introduced recently in \citep{CGGW}; see \autoref{T1}.

\medskip\noindent
The first main tool we will use are ribbons. These objects are highly useful when conjugating parabolic subgroups and we introduce them in \myref{preliminaries}{Section}. The other main result consists in making depend conjugacy stability of standard parabolic subgroups on a special property that we will call Property~$\star$. This property and its implications will be explained in \myref{s:prop}{Section}. Finally, in \myref{s:proof}{Section} we finish the proof of \autoref{mainT}.

\section{Garside elements and ribbons}\label{preliminaries}

Given a group $G$ and $g,x\in G$, we denote
by $x^g=g^{-1}xg$ the conjugate of $x$ by~$g$; this defines a right-action of $G$ on itself. 
In the same way, for $g\in G$ and a subset $H$ of $G$, we denote by $H^g$ the set of $g$-conjugates of elements of $H$.

\medskip\noindent
For the remainder of the present section, we fix an irreducible Artin-Tits group of spherical type $A_S$. The monoid $A_S^+$ consisting of \emph{positive} elements (which can be written as words on $S$ with only positive exponents) is a \emph{Garside monoid} (see \cite{BS,DP}): this involves, among other things, a \emph{fundamental} or \emph{Garside element} which we denote by $\Delta_S$ (for $X\subset S$, the Garside element of $A_X$ will be denoted by $\Delta_X$).

\medskip\noindent
%Among all Artin-Tits groups, those of spherical type are characterized by the existence of a finite-type Garside structure --\citep{BS,DP}.
%The \emph{classical Garside structure} of $A_S$ consists of 2 pieces of data: the \emph{positive monoid} $A_S^{+}$ (those elements of $A_S$ which can be written as words on $S$ with only positive exponents) and a \emph{fundamental (Garside) element} $\Delta_S\in A_S^{+}$ whose positive prefixes/suffixes form a lattice generating set of the group. If $X\subset S$, we will denoted the Garside element of the standard parabolic subgroup $A_X$ by $\Delta_X$. 

\medskip\noindent
{\bf Example.}
In the braid group on $n+1$ strands (Artin-Tits group of type $A_n$), the Garside element is $s_1(s_2s_1)(s_3s_2s_1)\cdots (s_ns_{n-1}\cdots s_1)$; it can be seen as a half-twist of the trivial braid on $n+1$ strands.

\medskip\noindent
Although the paper builds on previous works which use in a crucial way the Garside structure of~$A_S$, our arguments do not directly involve this structure so we only record some useful properties of the Garside element~$\Delta_S$.

\medskip\noindent
It is known that conjugation by~$\Delta_S$ is an involution and that $S^{\Delta_S}=S$. Moreover, $\Delta_S$ is central if $A_S$ is of type $A_1$, $B_n$, $D_n$ ($n$ even), $E_7$, $E_8$, $F_4$, $H_3$, $H_4$ or $I_{2m}$ ($m$ even) \citep{BS,Deligne}. \autoref{Table1} synthesizes the conjugacy action by $\Delta_S$ in the other irreducible cases. In each of the cases considered in \autoref{Table1}, we call \emph{flip automorphism} the inner automorphism $s\mapsto s^{\Delta_S}$ of~$A_S$. For more information about the specific construction of~$\Delta_S$, see \citep{BS}.

\begin{table}[h]
\begin{center}
\begin{tabular}{|c|c|c|c|c|}
\hline
$A_S$ & $A_n$ ($n\geqslant 2$) & $D_n$ ($n$ odd) & $E_6$ & $I_{2m}$ ($m$ odd) \\
\hline
$s_i^{\Delta_S}$ & $s_{n-i+1}$, $1\leqslant i \leqslant n$ & 
 $\begin{cases}
 s_2 & i=1 \\
 s_1 & i=2 \\
s_i & 3\leqslant i\leqslant n
\end{cases}$
 &  $\begin{cases}
 s_6 & i=1\\
 s_2 & i=2 \\
s_5 & i=3\\
s_4 & i=4 \\
s_3 & i=5 \\
s_1 & i=6\\
\end{cases}$ & 
$ s_{{(i+1)} \ mod\ 2}$
  \\
 \hline
\end{tabular}
\caption{Conjugation by the special Garside element $\Delta_S$.}\label{Table1}
\end{center}
\end{table}

\medskip
\noindent
We are now able to define ribbons. Note that the following definition is slightly different from the original definition of ribbon from \citep{Godelle} based upon \citep{ParisParabolicos}.

\begin{definition}\cite[Definition~4.1]{CGGW}
Let~$A_S$ be an Artin--Tits group of spherical type.
Given $X\subseteq S$ and~$t\in S$, we define the positive element
$$
  r(t,X)=\Delta_{X}^{-1}\Delta_{X\cup\{t\}}
$$ and we call it a \emph{ribbon}. If moreover $t$ is adjacent to $X$ in the Coxeter graph $\Gamma_S$, we say that 
$r(t,X)$ is an \emph{adjacent ribbon}.
\end{definition}

\begin{remark}
Notice that $r(t,X)$ conjugates $X$ to some subset ${X'}$ of ${X\cup\{t\}}$ and $X'=X^{\Delta_{X\cup\{t\}}}$. 
\end{remark}

\bigskip\noindent
The forthcoming proofs use a weak version of a result from \citep{AntolinCumplido}. The \emph{support} of a positive element $g$ of $A_S$ is defined as 
$$Supp(g)=\{s\in S, s\ {\text{appears in any positive word on $S$ representing $g$}}\}.$$ 

\begin{lemma}[{\citealp[Lemma 21]{AntolinCumplido}}]\label{counterexamples}
Let $g,h$ be positive elements of an Artin--Tits group $A_S$ of spherical type such that  $Supp(g)=Y\subsetneq S$ and $Supp(h)=Z\subsetneq S$. If~$g$ and~$h$ are conjugate in $A_S$, then there are 
subsets $Y=Y_0,\ldots,Y_{n}=Z$ of $S$ and adjacent ribbons $r(t_i,Y_{i-1})$ ($i=1,\ldots,n$) conjugating ${Y_{i-1}}$ to ${Y_i}$.
\end{lemma}

%\begin{remark}\label{resultribbons}
%The strong version of \autoref{counterexamples} says that any two parabolic subgroups $A_X$, $A_{X'}$ are conjugate in an Artin--Tits group of spherical type, we can find a conjugating element of the form $v=\alpha\beta$ where $\alpha\in A_X$ and $\beta$ is a composition of adjacent ribbons sending $X$ to ${X'}$. If one wants $v$ to induce a graph automorphism, then $\alpha$ has to be either the identity or an element inducing flip homomorphisms in some connected components of $\Gamma_{X}$, that is, the product of the special Garside elements corresponding to the flipped components. 
%\end{remark}

\section{The Property~$\star$}\label{s:prop}

\noindent
In this section we introduce our Property $\star$ and we show its sufficiency for conjugacy stability, in the spherical case. 
In a second step, we show that Property $\star$ holds in several cases. 
%Fix an Artin-Tits system $(A_S,S)$ of spherical type with Coxeter graph $\Gamma=\Gamma_S$ and $\emptyset\neq X\subsetneq S$. 

\begin{definition}
Let $(A_S,S)$ be an Artin-Tits system (of spherical type) and let $\emptyset\neq X\subsetneq S$
We say that $(A_X,A_S)$ \emph{satisfies Property $\star$}
if for all $Y_1,Y_2\subset X$, and $g\in A_S$ such that ${Y_1}^g = {Y_2}$, there exists $h\in A_X$ such that $s^h = s^g$ for all $s \in Y_1$.
\end{definition}

\begin{proposition}\label{P:Star2}
Let $(A_S,S)$ be an Artin-Tits system of spherical type and let $\emptyset \neq X\subsetneq S$. 
If $(A_X,A_S)$ has Property~$\star$, then $A_X$ is conjugacy stable in~$A_S$. 
\end{proposition}

\medskip\noindent
Before proceeding to the proof, we recall the important and recently defined notion of \emph{parabolic closure} of elements of Artin-Tits groups of \emph{spherical type}:

\begin{theorem}[{\citealp[Section 7, Lemma 8.1]{CGGW}}]\label{T1} 
Let $(A_S,S)$ be an Artin-Tits system of spherical type. 
For each $a\in A_S$, there is a unique minimal (with respect to inclusion) parabolic subgroup $P_a$ of $A_S$ which contains~$a$; we call this subgroup \emph{the parabolic closure of $a$}. Furthermore, for $g\in A_S$, we have ${P_a}^g=P_{a^g}$.
%i.e {\color{red}$\rho_g(P_a) = P_{\rho_g(a)}$}.
\end{theorem}

\noindent
\emph{Proof of \autoref{P:Star2}}.
Let $a,b\in A_X$ and $c\in A_S$ satisfying $a^c=b$. According to \autoref{T1},  we have that ${P_a}^c = P_{a^c}=P_b$.  According to \cite[Theorem~3]{cumplido}, both subgroups $P_a$ and~$P_b$ of~$A_X$ can be \emph{standardized} inside $A_X$: i.e. there exist $\alpha,\beta\in A_X$ and subsets $Y_a,Y_b$ of~$X$ such that ${P_a}^{\alpha} = A_{Y_a}$ and ${P_b}^{\beta} = A_{Y_b}$. Notice that
$A_{Y_a}^{\alpha^{-1}c\beta}=A_{Y_b}$. 

\medskip\noindent 
By \citep[Proposition 2.1.(3)]{Godelle} we can find $u\in A_S$ with ${Y_a}^u={Y_b}$ and $v\in A_{Y_b}$, such that $\alpha^{-1}c\beta = uv$. By Property~$\star$, we can find $u'\in A_X$ such that $s^{u'} = s^u$ for every $s\in Y_a$. It follows that $s^{u'v} = s^{uv} = s^{\alpha^{-1}c\beta}$ for all $s\in Y_a$; therefore for any element $z\in A_{Y_a}$, we have $z^{u'v}=z^{\alpha^{-1}c\beta}$. Applying this to the particular element $a^{\alpha}\in A_{Y_a}$, we obtain $a^{\alpha u'v} = a^{\alpha\alpha^{-1}c\beta}=a^{c\beta}=b^{\beta}$. It follows that $b=a^{\alpha u' v\beta^{-1}}$, and we note that $\alpha u' v\beta^{-1}\in A_X$. \hfill $\square$

\medskip

\begin{remark}
{Given an Artin-Tits system $(A_S,S)$ of spherical type, Property $\star$ for the pair $(A_X,A_S)$ implies that the automorphisms of $A_X$ induced by conjugation by an element in the normalizer $N_{A_S}(A_X)$ are inner automorphisms of $A_X$. 
Indeed, we know \citep[Theorem 0.1]{Godelle} that $N_{A_S}(A_X)=A_X\cdot QZ_{A_S}(A_X)$ (where $QZ_{A_S}(A_X)=\{g\in A_S, \ X^g=X\}$); Property~$\star$ then says that for $g\in QZ_{A_S}(A_X)$, we can find $h\in A_X$ such that $x^g=x^h$ for all $x\in A_X$ and the claim follows. }
\end{remark}

\bigskip

%\subsection{Checking Property~$\star$}\label{star}
\noindent
Now, we will see that Property~$\star$ holds in some cases.  %We will see later that these are exactly the non-listed cases in \autoref{mainT}(1).
%Before we start doing this, we may notice that every automorphism of a Coxeter graph of spherical type is composed essentially of the three following operations: Applying a symmetry to a connected component (this is performed by a flip automorphism); permuting the connected components of the graph; translating a connected component without decreasing the total number of connected component of the graph. 

\medskip

\begin{lemma}\label{L:MoreThanStar}
\begin{itemize}
    \item[(i)] Let $A_X$ be an Artin--Tits group of type $A_n$, $E_6$ or $I_{2\cdot 5}$ and let $\Gamma_X$ be its defining Coxeter graph. Let $Y_1,Y_2\subset X$, let $\Gamma_{Y_1}$, $\Gamma_{Y_2}$ be the respective induced subgraphs of $\Gamma_X$ and let $\psi:\Gamma_{Y_1}\longrightarrow \Gamma_{Y_2}$ be an isomorphism of labeled graphs.
Then there exists $v\in A_X$ such that $\psi(y)=y^v$, for all $y\in Y_1$. 
\item[(ii)] Let $A_S$ be any Artin--Tits group, let $\emptyset\neq X\subsetneq S$ such that $A_X$ is of type $A_n$, $E_6$ or $I_{2\cdot 5}$. Then $(A_X,A_S)$ has Property~$\star$.
\end{itemize}
\end{lemma}

\begin{proof}
Let us first prove (ii) as a consequence of (i). Let $Y_1,Y_2\subset X$ and $w\in A_S$ such that $Y_1^w=Y_2$; in particular, conjugation by~$w$ induces an isomorphism between the labeled graphs~$\Gamma_{Y_1}$ and~$\Gamma_{Y_2}$  and it follows from (i) that we can find $v\in A_X$ so that $y^v=y^w$, for all $y\in Y_1$. 

\medskip\noindent
Let us prove (i). Suppose first that $A_X$ is of type $I_{2\cdot 5}$. Write $\bar{s_i}=s_{(i+1)\ mod\ 2}$, for $i=1,2$. Then 
$\psi(y)=y$ for all $y\in Y_1$, or $\psi(y)=\bar y$ for all $y\in Y_1$. It then suffices to take $v$ to be the identity or $\Delta_X$, accordingly.

\medskip\noindent
Suppose that $A_X$ is of type~$A_n$. The graph $\Gamma_{Y_1}$ is a disjoint union of path graphs (possibly with a single vertex) and the graph isomorphism $\psi$: $\Gamma_{Y_1}\longrightarrow \Gamma_{Y_2}$ can be realized conjugating first by a product of ribbons (as in the example of \autoref{Figura_ribbons}) and then by a product of the Garside elements of some irreducible components of~$A_{Y_2}$. 

%
%
%Then $v=\alpha\beta$, where $\alpha$ is the product of the special Garside elements 
%corresponding to some connected component of $\Gamma_{X}$ and $\beta$ is a product of ribbons that can be seen as an element in the symmetric group with $n+1$ elements (see \autoref{resultribbons} and \autoref{Figura_ribbons}).
%

\begin{figure}[h]
\centering
\includegraphics[scale=0.6]{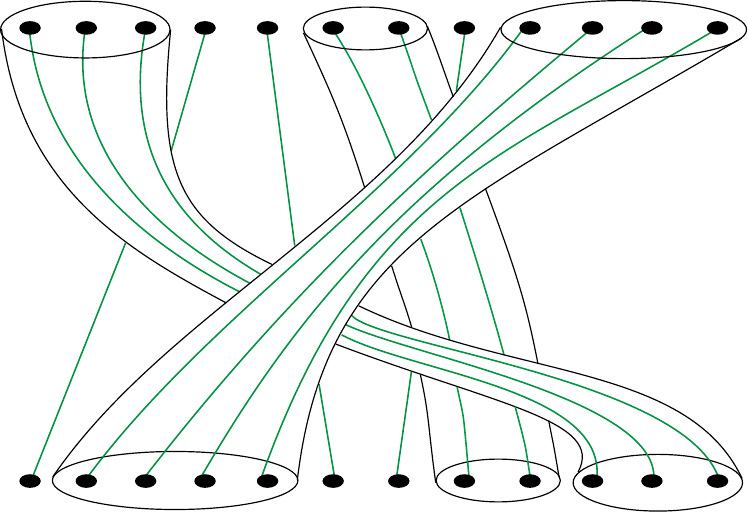}
\caption{In the braid group $\mathcal B_{12}$, a braid made of ribbons conjugating $Y_1=\{s_1,s_2\}\cup\{s_6\}\cup\{s_9,s_{10},s_{11}\}$ to $Y_2=\{s_2,s_3,s_4\}\cup\{s_8\}\cup\{s_{10},s_{11}\}$.}\label{Figura_ribbons}
\end{figure}

\medskip\noindent
Now suppose that~$A_X$ is of type~$E_6$. It can be checked that in this case, two standard parabolic subgroups $A_{Y_1}$ and $A_{Y_2}$ are isomorphic if and only if $Y_1$ and $Y_2$ are conjugate; therefore under our hypothesis $Y_1$ and $Y_2$ must be conjugate in $A_X$.

\medskip
\noindent
If $Y_1$ or $Y_2$ is a subset of $\{s_1,s_3,s_4,s_5,s_6\}$, we are back to the previous case, as $A_{\{s_1,s_3,s_4,s_5,s_6\}}$ is a braid group on 6 strands. 
In \autoref{Table2}, we list the conjugacy classes of subsets of $X$ with no representative in $\{s_1,s_3,s_4,s_5,s_6\}$. In each case, we see that every possible graph automorphism of $\Gamma_{Y_1}$ (the table considers generators of the automorphism group of $\Gamma_{Y_1}$) is induced by  conjugation by some $v\in A_X$, given explicitly in the last column. \end{proof}

% five cases and the corresponding form of~$v$ to induce the automorphism~$\psi$.
%%In the rest of the cases, the reader can check the one can find a ribbon $r=r(t,X)$ such that $t$ is adjacent to the connected component of $Y_1$ containing $s_2$ and $Y_1^r\subset \{s_1,s_3,s_4,s_5,s_6\}$. Then, the problem would reduce again to the braid case.
% \usepackage{multirow}

\begin{table}[h]
 \caption{The conjugacy classes of subsets of $X$ with no representative lying in $\{s_1,s_3,s_4,s_5,s_6\}$.}\label{Table2}
\begin{center}
\resizebox{\textwidth}{!}{
\begin{tabular}{|c|c|c|c|}
\hline
$Y_1$ & Type of $Y_1$ & Automorphism of $\Gamma_{Y_1}$ & $v$\\
\hline	
$\{s_1,s_2,s_4,s_6\}$  & $A_1\times A_1\times A_2$  &  \begin{tabular}{c} flip on $A_2$-component \\ transposition of the $A_1$-components \end{tabular}&  \begin{tabular}{c} $\Delta_{\{s_2,s_4\}}$ \\ $\Delta_X$ \end{tabular}\\
\hline
$\{s_1,s_2,s_3,s_5,s_6\}$ & 
$A_1\times A_2\times A_2$ &
\begin{tabular}{c} flip on the first $A_2$-component \\ 
flip on the second $A_2$-component \\ 
transposition of the $A_2$-components\end{tabular} & 
\begin{tabular}{c} 
$\Delta_{\{s_1,s_3\}}$ \\  
$\Delta_{\{s_5,s_6\}}$\\  
$\Delta_X$ 
\end{tabular}\\
\hline
 $\{s_2,s_3,s_4,s_5\}$ & $D_4$ & 
 \begin{tabular}{c} transposition of the leaves $s_3,s_5$ of $\Gamma_{Y_1}$ \\  transposition of the leaves $s_2,s_5$ of $\Gamma_{Y_1}$ \end{tabular} & 
 \begin{tabular}{c}
$\Delta_X$\\  $\Delta_{\{s_1,s_2,s_3,s_4,s_5\}}$\\
\end{tabular}\\
\hline
 $\{s_1,s_2,s_4,s_5,s_6\}$ & $A_1\times A_4$ & 
flip on the $A_4$-component & $\Delta_{\{s_2,s_4,s_5,s_6\}}$ \\
\hline 
 $X$ & $E_6$ & flip of the whole $\Gamma_X$ & $\Delta_X$  \\
\hline
\end{tabular}}
\end{center}
\end{table}

\begin{remark}
Although this will not be used in the sequel, we note that the statement of \autoref{L:MoreThanStar} holds as well for $A_X$ of type $I_{2\cdot m}$, $m$ odd and $E_8$. 
\end{remark}

\begin{lemma}\label{L:BD}
Let $A_S$ be an irreducible Artin--Tits group of spherical type and let $\Gamma_S$ be its defining Coxeter graph. Let $\emptyset \neq X\subsetneq S$. Assume that: 
\begin{itemize}
    \item $A_X$ is of type $B_n$, or 
    \item $A_X$ is of type $D_n$ ($n$ odd) and $A_S$ is of type $D_m$ ($m>n$). 
    \end{itemize}
    Then $(A_X,A_S)$ satisfies Property $\star$. 
\end{lemma}

\begin{proof} Whenever $Z\subset S$, we write $\Gamma_Z$ for the subgraph of $\Gamma_S$ induced by $Z$. 
We fix once and for all $Y_1, Y_2\subset X$ and $w\in A_S$ such that $Y_1^w=Y_2$. 
Suppose first that~$A_X$ is of type~$B_n$; observe that~$A_S$ must be of type $F_4$ or $B_m$ ($m>n$) and that the first possibility might occur only if $n\leqslant 3$.
We give a detailed proof assuming that~$A_S$ is of type~$B_m$; the case~$A_S$ of type~$F_4$ can be dealt with in a similar fashion and is left as an exercise for the reader. 

\medskip\noindent
Given any subset $Z\subset S$, we denote by~$Z'$ the set of vertices of the connected component of~$\Gamma_Z$ containing $s_1$ and we set $Z'=\emptyset$ if $s_1\notin Z$; we also denote $Z''=Z\setminus Z'$. 
Notice that the conjugation by~$w$ induces an isomorphism between the Coxeter graphs~$\Gamma_{Y_1}$ and~$\Gamma_{Y_2}$. 
Then $Y'_1=Y'_2$ with $y^w=y$ for all $y\in Y'_1$ (due to the defining relations of~$A_S$) and the graphs~$\Gamma_{Y''_1}$ and~$\Gamma_{Y''_2}$ have to be isomorphic. If $Y_1''$ is empty, we can replace $w$ by the trivial element of $A_X$ and we are done. Otherwise observe that ${Y''_1}$ and~${Y''_2}$ are subsets of $\{s_2,\ldots, s_n\}$, which generates a braid group on $n$ strands. By applying \autoref{L:MoreThanStar}, we can find $w'\in A_{\{s_2,\ldots,s_n\}}\subset A_X$ performing the same conjugation as~$w$ on~$Y_1$ (if $Y'_1\neq \emptyset$, we can choose $w'\in A_{\{s_3,\ldots,s_n\}}$ commuting with $Y'_1$).

\medskip\noindent
Suppose now that~$A_X$ is of type~$D_n$ ($n$ odd) and that~$A_S$ is of type $D_m$ (${m>n}$). 
Recall (\autoref{Table1}) that conjugation by $\Delta_X$ leaves invariant $s_3,\ldots, s_n$ and permutes~$s_1$ and~$s_2$. If each of the chosen subsets~$Y_1$ and~$Y_2$  contains at most one of $s_1,s_2$, up to replacing one of $Y_i$ ($i=1,2$) by~
$Y_i^{\Delta_X}$, we may assume that both $Y_1, Y_2$ are subsets of $\{s_1,s_3,\ldots, s_n\}$; the latter set defines a braid group of type~$A_{n-1}$ and  \autoref{L:MoreThanStar} allows us to conclude.

\medskip\noindent
Suppose that~$Y_1$ contains both $s_1$ and $s_2$. Then $Y_2$ has to contain also both~$s_1$ and~$s_2$. To see this, observe that the only ribbon 
adjacent to $\{s_1,s_2\}$ is $s_3s_1s_2s_3$ which conjugates~$s_1$ to~$s_2$ and~$s_2$ to~$s_1$ and apply \autoref{counterexamples}: it follows that~$s_1$ and~$s_2$ can be simultaneously conjugated in~$A_S$ to letters in~$S$ only if they are fixed or permuted with each other. As for type~$B$, denoting by~$Y'_i$ ($i=1,2$) the set of vertices of the (union of the) connected component(s) of~$\Gamma_{Y_i}$ containing~$s_1$ and~$s_2$, we obtain $Y'_1=Y'_2$. We also see that the $Y''_i=Y_i\setminus Y'_i$ define isomorphic subgroups of the braid group $A_{\{s_4,\ldots, s_n\}}$, and we conclude as in type~$B$ case using \autoref{L:MoreThanStar} again.
%
%If $\Gamma_{Y_1}$ has a connected component of type~$D$ including the vertices $\{s_1,s_2,s_3,s_4\}$, then $\Gamma_{Y_2}$ has to have the same connected component in order to have an isomorphism. If $\Gamma_{Y_1}$ has a connected component of type $A_3$ having as vertices $\{s_1,s_2,s_3\}$. The only adjacent ribbon $r(t, \{s_1,s_2,s_3\})$ fixes $\{s_1,s_2,s_3\}$, hence, by \autoref{counterexamples}, $\Gamma_{Y_2}$ has to have the same connected component also in this case. Now let $s_1,s_2\in Y_1$ and $s_3\notin Y_1$.  Suppose that neither~$s_1$ nor~$s_2$ can be conjugate to $Y_1\setminus \{s_1,s_2\}$. Then, also  $s_1,s_2\in Y_1$ and $s_3\notin Y_1$. In the three previous cases, the components containing~$s_1$ and~$s_2$ are exactly the same for $Y_1$ and $Y_2$. Therefore, one can follow the same reasoning as for type~$B_m$ and have Property~$\star$.
%
%\medskip\noindent
%Finally, let $s_1,s_2\in Y_1$ and $s_3\notin Y_1$ and suppose that $s_1$ (or $s_2$) can be conjugate to an element $s'\in Y_1\setminus \{s_1,s_2\}$. Since $s_1$ and $s'$ lie in a braid group of type $A_{n-1}$, we know by \autoref{L:MoreThanStar} that this conjugation can be made inside~$A_X$. Therefore, up to conjugation in $A_X$ we can suppose that both $Y_1$ and $Y_2$ lie in the same braid group and again by \autoref{L:MoreThanStar} we will have Property~$\star$.
\end{proof}

\begin{remark}
It is easily seen that the statement of \autoref{L:MoreThanStar} is not true if $A_X$ is of type $B$ or $D$. If $A_X$ is of type $B$, it suffices to consider $Y_1=Y_2=\{s_1,s_2\}$ and the automorphism~$\psi$ of~$\Gamma_{Y_1}$ which permutes its two vertices; it can be checked that $s_1$ and $s_2$ are not conjugate in~$A_X$ (see \autoref{L:Odd}), so $\psi$ fails to be induced by an inner automorphism of~$A_X$. If~$A_X$ has type~$D$, choosing $Y_1=\{s_1,s_2\}$ and $Y_2=\{s_1,s_4\}$, we've just seen in the above proof that the graph isomorphism $\Gamma_{Y_1}\longrightarrow \Gamma_{Y_2}$ given by $s_1\mapsto s_1$ and $s_2\mapsto s_4$ is not induced by any inner automorphism of~$A_X$. 
\end{remark}

\section[Proof of Theorem~1]{Proof of \autoref{mainT}}\label{s:proof}

\subsection{Irreducible case}

Let $A_S$ be an irreducible Artin-Tits group of spherical type and let $\Gamma_S$ be its defining Coxeter graph. Vertices of $\Gamma_S$ are numbered $s_1,\ldots s_{\#S}$, according to \autoref{F:Coxeter}. Let $\emptyset\neq X\subsetneq S$ such that $A_X$ is irreducible.  First, we observe that, as~$A_X$ is a {proper} subgroup of~$A_S$, it cannot be of type $E_8,F_4,H_4$ or $I_{2m}$, $m>5$. 

\medskip\noindent
Suppose that the pair $(A_X,A_S)$ does not satisfy any of the conditions (a) to (e) of \autoref{mainT}(1). The group $A_X$ cannot be either of type $E_7,D_{2k}$ or $H_3$; otherwise $(A_X,A_S)$ would satisfy either (c) (d) or (e). Finally, $A_X$ can be of type $D_5$ ($D_7$, respectively) only if~$A_S$ is of type~$D_n$, $n\geqslant 6$, ($n\geqslant 8$, respectively); otherwise $(A_X,A_S)$ would satisfy (a) or (b). Then \autoref{P:Star2}, \autoref{L:MoreThanStar} and \autoref{L:BD} show that $A_X$ is conjugacy stable in $A_S$, as desired.

\medskip\noindent
Therefore, to prove the first part of \autoref{mainT}, one has to prove that $A_X$ is not conjugacy stable in~$A_S$ whenever $(A_X,A_S)$ satisfies one of the conditions (a) to (e). 
\autoref{counterexamples} will be the main tool to provide counterexamples. In each case (a) to (e), we shall exhibit two elements of $A_X$ which are conjugate in~$A_S$ but not in $A_X$. 
%We are going to consider pairs $(\Gamma_X,\Gamma_S)$ of Coxeter graphs with a given explicit type, where $\Gamma_X$ is an induced subgraph of $\Gamma_S$. Vertices of $\Gamma_S$ are numbered according to \autoref{F:Coxeter}. 

\medskip\noindent
Let $\Gamma_X$ be the defining Coxeter graph of $A_X$ and number the elements of $X$ $x_1,\ldots, x_{\#X}$, according to  \autoref{F:Coxeter}. 
Notice that there might be different ways to embed~$\Gamma_X$ as an induced subgraph of~$\Gamma_S$; following our notation, this is to say that a given~$x_i$ may be 
equal to distinct~$s_j$'s, depending on the chosen embedding of $\Gamma_X$ in~$\Gamma_S$.

\begin{itemize}
\item[(a)] Suppose that $A_X$ is of type $D_5$ and $A_S$ is of type $E_6,E_7$ or 
$E_8$. 
There are 4 different embeddings $\iota: \Gamma_X\hookrightarrow \Gamma_S$, namely:
\begin{center}
$\iota_1 : x_1=s_2$, $x_2=s_3$, $x_3=s_4$, $x_4=s_5$, $x_5=s_6$, \\
$\iota_2 :x_1=s_3$, $x_2=s_2$, $x_3=s_4$, $x_4=s_5$, $x_5=s_6$,\\
$\iota_3: x_1=s_5$, $x_2=s_2$, $x_3=s_4$, $x_4=s_3$, $x_5=s_1$,\\
$\iota_4:x_1=s_2$, $x_2=s_5$, $x_3=s_4$, $x_4=s_3$, $x_5=s_1$.
\end{center}
However to pass from one to another it suffices to pre- or post-compose by graph automorphisms which are induced by conjugation by $\Delta_X$ or $\Delta_{\{s_1,s_2,s_3,s_4,s_5,s_6\}}$ respectively. Therefore it is enough to give a pair of non-conjugate elements of $A_X$ whose images under~$\iota_1$ are conjugate elements of $A_S$. 

Take $g=\iota_1(x_1x_3x_2)=s_2s_4s_3$ and $h= \iota_1(x_4x_3x_2)=s_5s_4s_3$. 
%; let $Y=Supp(g)$ and $Z=Supp(h)$. We then have $Y=\{s_2, s_3, s_4\}$ and $Z=\{s_3,s_4,s_5\}$. 
The following product of ribbons (each arrow indicates the conjugation by its label) conjugates $g$ to $h$ in $A_S$:

\smallskip

\begin{center}
\tikzset{main node/.style={rectangle,rounded corners=.8ex,draw,minimum size=0.8cm,inner sep=0pt},}
 \begin{tikzpicture}
    \tikzstyle{flecha}=[->, thick,>=latex]
    \node[main node] (1) {$\,Y=\{s_2,s_3,s_4\}\,$};
    \node[main node] (2) [right = 1.8cm  of 1]  {$\,{Y_2}=\{s_1,s_3,s_4\}\,$};
    \node[main node] (3) [right = 1.8cm  of 2] {$\,{Z}=\{s_3,s_4,s_5\}\,$};

    \draw[flecha] (1)to[left] node[above]{\small $r( s_1,Y)$}(2);
    \draw[flecha] (2)to[left] node[above]{$r(s_5,Y_2)$}(3);
   %\draw[flecha] (3)to[left] node[above]{$r(s_6,Y_3)$}(4);

\end{tikzpicture}
\end{center}

\smallskip

However, the only vertex in $\Gamma_X$ which is adjacent to $\{x_1,x_3,x_2\}$
 is $x_4$ and we observe that $r(x_4,\{x_1,x_3,x_2\})=x_4x_3x_1x_2x_3x_4$ normalizes~$A_{\{x_1,x_3,x_2\}}$. 
 Therefore, by \autoref{counterexamples}, $x_1x_3x_2$ and $x_4x_3x_2$ cannot be conjugate in~$A_X$. 
 
\medskip

\item[(b)] Suppose that $A_X$ is of type $D_7$ and $A_S$ is of type $E_8$. 
There is only one induced subgraph of type $D_7$ in $\Gamma_S$ and two ways of embedding it: 
\begin{center}
$\iota_1: x_1=s_2$, $x_2=s_3$, $x_3=s_4$, $x_4=s_5$, $x_5=s_6$, $x_6=s_7$, $x_7=s_8$, \\
$\iota_2: x_1=s_3$, $x_2=s_2$, $x_3=s_4$, $x_4=s_5$, $x_5=s_6$, $x_6=s_7$, $x_7=s_8$,
\end{center}

which differ by precomposing by the graph automorphism of~$\Gamma_X$ induced by conjugation by~$\Delta_X$. 

Take $g=\iota_1(x_1 x_3 x_2)=s_2s_4s_3$ and $h= \iota_1(x_4x_3x_2)=s_5s_4s_3$. We conclude exactly in the same way as in (a): $g$ and $h$ are conjugate in $A_S$ but the only vertex $t$ of $\Gamma_X$ which is adjacent to $\{x_1,x_3,x_2\}$ produces an adjacent ribbon 
$r(t,\{x_1,x_3,x_2\})$ which normalizes~$A_{\{x_1,x_3,x_2\}}$. Therefore by \autoref{counterexamples}, $x_1x_3x_2$ and $x_4x_3x_2$ are not conjugate in $A_X$. 

\medskip

\item[(c)] Suppose that $A_X$ is of type~$E_7$ and~$A_S$ is of type~$E_8$. 
We must have $x_i=s_i$ for all $1\leqslant i\leqslant 7$.  
Take $g=s_1 s_3 s_4 s_5s_6$ and $h= s_2s_4s_5s_6s_7$.
%We have $Y=\{s_1, s_3, s_4,s_5,s_6\}$ and $Z=\{s_2,s_4,s_5,s_6,s_7\}$. 
The following product of ribbons conjugates $g$ to $h$ in~$A_S$: 

\smallskip

\begin{center}
\tikzset{main node/.style={rectangle,rounded corners=.8ex,draw,minimum size=0.8cm,inner sep=0pt},}
 \begin{tikzpicture}
    \tikzstyle{flecha}=[->, thick,>=latex]
    \node[main node] (1) {$\,Y=\{s_1, s_3, s_4,s_5,s_6\}\,$};
    \node[main node] (2) [below = 0.8cm  of 1]  {$\,{Y_2}=\{s_3,s_4,s_5,s_6,s_7\}\,$};
    \node[main node] (3) [right = 1.8cm  of 2] {$\,{Y_3}=\{s_4,s_5,s_6,s_7,s_8\}\,$};
    \node[main node] (4) [above = 0.8cm  of 3] {$\,Z=\{s_2,s_4,s_5,s_6,s_7\}\,$};

    \draw[flecha] (1)to[left] node[left]{ $r(s_7,Y)$}(2);
    \draw[flecha] (2)to[left] node[above]{$r(s_8,Y_2)$}(3);
    \draw[flecha] (3)to[left] node[right]{$r(s_2,Y_3)$}(4);

\end{tikzpicture}
\end{center}

\smallskip

However, a conjugation in $A_X$ by a sequence of adjacent ribbons never takes $\{x_1,x_3,x_4,x_5,x_6\}$ to $\{x_2,x_4,x_5,x_6,x_7\}$, as shows the following
picture: 

\smallskip

\begin{center}
\tikzset{main node/.style={rectangle,rounded corners=.8ex,draw,minimum size=0.8cm,inner sep=0pt},}
 \begin{tikzpicture}
    \tikzstyle{flecha}=[->, thick,>=latex]
    \node[main node] (2) {$\,Y=\{x_1,x_3,x_4,x_5,x_6\}\,$};
    \node[main node] (3) [right = 1cm  of 2] {$\,{Y_2}=\{x_3,x_4,x_5,x_6,x_7\}\,$};

    \draw[flecha] (2)to[bend left] node[above]{$r( x_7,Y)$}(3);
    \draw[flecha] (3)to[bend left] node[below]{ $r( x_1,Y_2)$}(2);
    \draw[flecha] (3)to[loop right] node[above right]{$r(x_2,Y_2)$}(3);
    \draw[flecha] (2)to[loop left] node[above left]{$ r(x_2,Y)$}(2);

\end{tikzpicture}
\end{center}

\smallskip

Hence, by \autoref{counterexamples}, $x_1 x_3 x_4 x_5 x_6$ and $x_2x_4x_5x_6x_7$ are not conjugates in $A_X$.

\medskip

\item[(d)] Suppose that $A_X$ is of type~$D_{2k}$. There exists $X\subsetneq X'\subseteq S$ so that~$A_{X'}$ is of type $D_{2k+1}$ so we can assume that~$A_S$ is of type~$D_{2k+1}$. 
We have two possible embeddings 
\begin{center}
$\iota_1$: $x_1=s_1$, $x_2=s_2$, $x_i=s_i$ for $3\leqslant i\leqslant 2k$,\\
$\iota_2$: $x_1=s_2$, $x_2=s_1$, $x_i=s_i$ for $3\leqslant i\leqslant 2k$
\end{center}
which differ by post-composing by the graph 
automorphism of $\Gamma_S$ induced by conjugation by $\Delta_S$.
Let $Y=\{s_1,s_3,\dots,s_{2k}\}$ and $Z=\{s_2,s_3,\dots,s_{2k}\}$. 

The product of adjacent ribbons $r_1=r(s_{2k+1},Y)$ and $r(s_2,Y^{r_1})$ conjugates $Y$ to $Z$ in $A_S$; it also conjugates the element $s_1s_3\cdots s_{2k}=\iota_1(x_1x_3\cdots x_{2k})$ to $s_2s_3\cdots s_{2k}=\iota_1(x_2x_3\ldots x_{2k})$. 
However, due to \autoref{counterexamples}, the two elements $x_1x_3\cdots x_{2k}$ and $x_2x_3\ldots x_{2k}$ cannot be conjugate inside the parabolic subgroup~$A_X$ because the only possible adjacent ribbon -- $r(x_2,\{x_1,x_3,\ldots,x_{2k}\})$ -- normalizes $A_{\{x_1,x_3,\dots,x_{2k}\}}$.

\item[(e)] Suppose that~$A_X$ is of type~$H_3$ and~$A_S$ is of type $H_4$. There is only one possible embedding and for $1\leqslant i\leqslant 3$, we have $x_i=s_i$. 
We are going to prove that $s_1s_3s_3$ and $s_3s_1s_1$ are conjugate in~$A_S$ but not in~$A_X$. 
One can easily verify that conjugation by $$r(s_4,\{s_1,s_3\}) \cdot r(s_2,\{s_1,s_4\})\cdot \Delta_{\{s_2,s_3,s_4\}}\cdot r(s_1,\{s_2,s_4\}) \cdot r(s_3,\{s_1,s_4\})$$ permutes $s_1$ and $s_3$ and hence conjugates $s_1s_3s_3$ to $s_3s_1s_1$. 

However, \autoref{counterexamples} shows that $x_1x_3x_3$ and $x_3x_1x_1$ are not conjugate in~$A_X$ because the adjacent ribbon $r(x_2,\{x_1,x_3\})$ commutes with~$x_1$ and~$x_3$. 
\end{itemize}

\noindent
This finishes the proof of the first part of \autoref{mainT}.

\subsection{Reducible case}
Let $A_S$ be an irreducible Artin-Tits group of spherical type and let $\emptyset\neq X\subsetneq S$ such that~$A_X$ is reducible. Let~$\Gamma_X$ be the subgraph of~$\Gamma_S$ induced by~$X$. 
We first make a preliminary observation. 
\begin{lemma}\label{L:Odd}
Let $(A_S,S)$ be any Artin-Tits system; let $\Gamma_S$ be the defining Coxeter graph. Two letters $s,t\in S$ are conjugate in $A_S$ if and only if the vertices $s$ and $t$ of the Coxeter graph $\Gamma_S$ can be connected in~$\Gamma_S$ by a path following only edges with odd labels (or no label).
\end{lemma}
\begin{proof}
Suppose that $s,s'$ are connected by an edge with odd label $m$ or no label, in which case we set $m=3$. We have $\Pi(s,s';m-1)s=s'\Pi(s,s';m-1)$ and $s,s'$ are conjugate. An immediate induction shows that $s,s'$ are conjugate in $A_S$ whenever they are connected in $\Gamma_S$ by a path following only edges with odd labels (or no label). 
Assume on the contrary that no path with this property connects $s$ and $s'$ in $\Gamma_S$. It follows from \citep[Chap. IV, \S1, no.3, Proposition 3]{Bourbaki} that the respective images of $s$ and $s'$ in the Coxeter group $A_S/\langle\langle s^2, s\in S\rangle\rangle$ are not conjugate; therefore $s$ and $s'$ cannot be conjugate either.  
\end{proof}

\noindent
The previous result implies that if $\Gamma_X$ has two connected components that can be connected through a path following only edges with odd labels (or no label) in $\Gamma_S$, then $A_X$ cannot be conjugacy stable in $A_S$. The only cases that do not satisfy this condition are the cases (a) and (b) of \autoref{mainT}(2). Therefore, to finish the proof of our theorem we just need to show that in these cases $A_X$ is conjugacy stable in~$A_S$.

\medskip\noindent
In both cases, we have $A_X=A_{X_1}\times A_{X_2}$, where $A_{X_1}$ is cyclic generated by a letter of $S$ which is conjugate to no other letter of $X$ and $A_{X_2}$ is a braid group.
%Then, any conjugate of any element $\alpha_1\alpha_2\in A_X$, with $\alpha_1\in A_{X_1}$ and $\alpha_2\in A_{X_2}$ can be written in the form $\alpha'_1\alpha'_2\in A_X$, where $\alpha'_1\in A_{X_1}$ is conjugate to $\alpha_1$ and $\alpha'_2\in A_{X_2}$ is conjugate to $\alpha_2$. 
By \autoref{L:MoreThanStar}, the pair $(A_X,A_S)$ has Property $\star$ and by \autoref{P:Star2}, $A_X$ is conjugacy stable in $A_S$. This completes the proof of \autoref{mainT}.

\medskip

\begin{remark}
A posteriori, one sees that, when $A_S$ is an Artin-Tits group of spherical type and $\emptyset\neq X\subset S$, $A_X$ is conjugacy stable in $A_S$ if and only if $(A_X,A_S)$ satisfies Property~$\star$.
\end{remark}

%%%%%%%%%%%%%%%%%%%%%%%%%%%%%%%%%%%%%%%%%%%%%%%%%%%%%%%%%%%%%%%%%%%%%%%%%%%%%%%%%

\bigskip
\noindent
{\bf Funding.}
This work was supported by Fondo Nacional de Desarrollo Cient\'ifico y Tecnol\'ogico [Iniciaci\'on 11140090 to M.Ca. and B.C.dlC., Regular 1180335 to M.Ca. and B.C.dlC.]; Comisi\'on Nacional de Investigaci\'on Cient\'ifica y Tecnol\'ogica [PIA ACT1415 to M.Ca., PAI 79160023 to M.Ca.];  Ministerio Espa\~nol de Ciencia y Competitividad [MTM2016-76453-C2-1-P to M.Ca. and M. Cu.]; Junta de Andaluc\'ia [FQM-218 to M.Ca. and M.Cu.]; Fondo Europeo de Desarrollo Regional to M. Ca and M.Cu; and Consejo Nacional de Ciencia y Tecnolog\'ia [Convocatoria de Investigaci\'on Cient\'ifica B\'asica 2016, Referencia: 284621 to B.C.dlC.]

\bigskip
\noindent
{\bf {Acknowledgements.}} {The first author is very grateful to Ivan Marin for suggesting the problem under study during the Conference SUMA 2016 in Valpara\'iso, Chile. He also thanks Jes\'us Juyumaya for suggesting interesting readings. All three authors thank Luis Paris for his interest and comments on the problem. We also thank the referee for reading this article and suggesting to simplify the proofs.}

\bibliography{BiblioConj}

\end{document}